\documentclass{amsart}
 \usepackage{amsmath,amssymb,amsfonts,amsthm,amscd,textcomp,times,booktabs,mathrsfs}
   \usepackage[dvips]{graphicx}
    \usepackage{braket}
  \usepackage{color,float}
  \usepackage{bm}
\usepackage[all]{xy}
\newtheorem{theorem}{Theorem}

\newtheorem{proposition}[theorem]{Proposition}
\newtheorem{lemma}[theorem]{Lemma}
\newtheorem{definition}[theorem]{Definition}
\newtheorem{corollary}[theorem]{Corollary}

\newtheorem{example}{Example}

\numberwithin{equation}{section}
\numberwithin{theorem}{section}

\newcommand{\Z}{\ensuremath{\mathbb{Z}}}

\newcommand{\R}{\ensuremath{\mathbb{R}}}
\newcommand{\C}{\ensuremath{\mathbb{C}}}
\newcommand{\I}{\ensuremath{\sqrt{-1}}}

\newcommand{\GL}{\ensuremath{\mathrm{GL}}}

\newcommand{\SO}{\ensuremath{\mathrm{SO}}}
\newcommand{\Spin}{\ensuremath{\mathrm{Spin(7)}}}
\newcommand{\SU}{\ensuremath{\mathrm{SU}}}

\newcommand{\Symp}{\ensuremath{\mathrm{Sp}}}
\newcommand{\Spec}{\ensuremath{\mathrm{Spec}}}

\newcommand{\del}{\ensuremath{\mathrm{\partial}}}

\newcommand{\der}{\ensuremath{\mathrm{d}}}

\newcommand{\norm}[1]{\ensuremath{\left| #1 \right|}}
\newcommand{\Norm}[1]{\ensuremath{\left\| #1 \right\|}}
\newcommand{\restrict}[2]{\ensuremath{\left. #1 \right|_{#2}}}

\DeclareMathOperator{\Real}{Re}

\DeclareMathOperator{\Hom}{Hom}

\DeclareMathOperator*{\Hol}{Hol}

\DeclareMathOperator{\Div}{Div}
\DeclareMathOperator{\divisor}{div}
\DeclareMathOperator{\Pic}{Pic}
\def\bM{\mathbf{M}}
\def\bN{\mathbf{N}}
\def\bbT{\mathbb{T}}
\def\quot{\:\:{/\hspace{-0.3cm}/}\:\:}
\def\a{\alpha}

\def\D{\Delta}

\def\sp{\hspace{0.3cm}}
\def\Ahat{\widehat{A}}
\def\P{\mathbb{P}}
\def\Cone{\mathrm{Cone}}
\def\PC{\mathrm{PC}}

\begin{document}
\title{Doubling construction of Calabi-Yau fourfolds 
\linebreak from toric Fano fourfolds}

\author{Mamoru Doi}

 \email{doi.mamoru@gmail.com}

\author{Naoto Yotsutani}

\address{School of Mathematical Sciences at Fudan University,
Shanghai, 200433, P. R. China}
\email{naoto-yotsutani@fudan.edu.cn}

 \dedicatory{}
\date{\today}
\maketitle

\noindent{\bfseries Abstract.}
We give a differential-geometric construction of Calabi-Yau fourfolds by the `doubling' method, 
which was introduced in \cite{DY14} to construct Calabi-Yau threefolds.
We also give examples of Calabi-Yau fourfolds from toric Fano fourfolds.
Ingredients in our construction are \emph{admissible pairs}, which were first dealt with by Kovalev in \cite{K03}.
Here in this paper an admissible pair $(\overline{X},D)$ consists of
a compact K\"{a}hler manifold $\overline{X}$ and
a smooth anticanonical divisor $D$ on $\overline{X}$.
If two admissible pairs $(\overline{X}_1,D_1)$ and $(\overline{X}_2,D_2)$ with $\dim_{\mathbb{C}}\overline{X}_i=4$ satisfy
the \emph{gluing condition}, we can glue $\overline{X}_1\setminus D_1$ and
$\overline{X}_2\setminus D_2$ together to obtain a compact Riemannian $8$-manifold $(M,g)$
whose holonomy group $\mathrm{Hol}(g)$ is contained in $\mathrm{Spin}(7)$.
Furthermore, if the $\widehat{A}$-genus of $M$ equals $2$, then $M$ is a Calabi-Yau fourfold, i.e.,
a compact Ricci-flat K\"{a}hler fourfold with holonomy $\mathrm{SU}(4)$.
In particular, if $(\overline{X}_1,D_1)$ and $(\overline{X}_2,D_2)$
are identical to an admissible pair $(\overline{X},D)$,
then the gluing condition holds automatically, so that we obtain a compact Riemannian 
$8$-manifold $M$ with holonomy contained in $\mathrm{Spin}(7)$.
Moreover, we show that if the admissible pair is obtained from \emph{any} of the toric Fano fourfolds,
then the resulting manifold $M$ is a Calabi-Yau fourfold by computing $\widehat{A}(M)=2$.

\section{Introduction}
The purpose of this paper is to give a gluing construction and examples
of Calabi-Yau fourfolds.
In our previous paper \cite{DY14}, we gave a construction of Calabi-Yau \emph{three}folds from two admissible pairs
$(\overline{X}_1,D_1)$ and $(\overline{X}_2,D_2)$ with $\dim_{\C} \overline{X}_i=3$, so that 
an admissible pair $(\overline{X},D)$ consists of a compact K\"{a}hler threefold $\overline{X}$ 
and an anticanonical $K3$ divisor $D$ on $\overline{X}$ (see Definition $\ref{def:admissible}$).
Also, we gave the `doubling' construction as follows: 
If two admissible pairs
$(\overline{X}_1,D_1)$ and $(\overline{X}_2,D_2)$ with $\dim_{\C} \overline{X}_i=3$
are identical to an admissible pair $(\overline{X},D)$,
then we can always construct a Calabi-Yau threefold by gluing together the two copies of $\overline{X}\setminus D$.
In this paper we shall apply this construction in complex dimension \emph{four}, 
so that we shall consider admissible pairs $(\overline{X},D)$ with $\dim_\C\overline{X}=4$ 
and $D$ an smooth anticanonical divisor on $\overline{X}$.
As in \cite{DY14}, we use Kovalev's gluing technique in \cite{K03}, 
which was used to construct compact $G_2$-manifolds. Also, we use Joyce's analysis
on $\Spin$-structures \cite{J00}, while we used in \cite{DY14} his analysis on $G_2$-structures.

In our construction, we begin with two admissible pairs $(\overline{X}_1,D_1)$
and $(\overline{X}_2,D_2)$ with $\dim_\C\overline{X}_i=4$ as above.
Then by the existence result of
an asymptotically cylindrical Ricci-flat K\"{a}hler form on $\overline{X}_i\setminus D_i$,
each $\overline{X}_i\setminus D_i$ has a natural asymptotically cylindrical
torsion-free $\Spin$-structure $\Phi_i$. 
Now suppose the \emph{asymptotic models} $(D_i\times S^1\times\R_+,\Phi_{i,\rm{cyl}})$
of $(\overline{X}_i\setminus D_i,\Phi_i)$ are isomorphic in a suitable sense,
which is ensured by the {\it gluing condition} defined later (see Section $\ref{sec:gluing_cond}$).
Then as in Kovalev's construction in \cite{K03}, we can glue together
$\overline{X}_1\setminus D_1$ and $\overline{X}_2\setminus D_2$ 
along their cylindrical ends $D_1\times S^1\times (T-1,T+1)$ and
$D_2\times S^1\times (T-1,T+1)$, to obtain a compact $8$-manifold $M_T$.
Moreover, we can glue together the torsion-free $\Spin$-structures 
$\Phi_i$ on $\overline{X}_i\setminus D_i$
to construct a $\der$-closed $4$-form $\widetilde{\Phi}_T$ on $M_T$, which is projected to a 
$\Spin$-structure $\Phi_T =\Theta (\widetilde{\Phi}_T)$
{\it with small torsion} for sufficiently large $T$.
Using the analysis on 
$\Spin$-structures by Joyce \cite{J00},
we shall prove that $\Phi_T$ can be deformed into a torsion-free $\Spin$-structure
for sufficiently large $T$, so that
the resulting compact manifold $M_T$ admits a Riemannian metric with holonomy contained in $\Spin$.
Since $M=M_T$ is simply-connected, the $\Ahat$-genus $\Ahat(M)$ of $M$ is $1,2,3$ or $4$, 
and the holonomy group is determined as $\Spin , \SU(4), \Symp (2), \Symp(1)\times \Symp(1)$ 
respectively (see Theorem $\ref{thm:A-hat}$).
Hence if $\Ahat(M)=2$, then $M$ is a Calabi-Yau fourfold.

For a given admissible pair $(\overline{X}_1, D_1)$, it is difficult to find a suitable
admissible pair $(\overline{X}_2, D_2)$ with $D_2$ isomorphic to $D_1$.
In the three-dimensional case, if $(\overline{X},D)$ is an admissible pair,
then $D$ is a $K3$ surface. Thus $D_1$ and $D_2$ are at least diffeomorphic
and so are the cylindrical ends of $\overline{X}_1\setminus D_1$ and $\overline{X}_2\setminus D_2$
which we glue together. 
Meanwhile in the four-dimensional case, 
the topological type of the Calabi-Yau divisor $D$ 
for an admissible pair $(\overline{X},D)$ varies with $\overline{X}$.
However, if $(\overline{X}_1, D_1)$ and $(\overline{X}_2, D_2)$ are identical to an
admissible pair $(\overline{X}, D)$, then the gluing condition holds automatically.
Therefore we can always 
construct a compact simply-connected Riemannian
$8$-manifold $(M,g)$ with $\Hol(g)\subseteq \mathrm{Spin}(7)$
by doubling an  
admissible pair $(\overline{X},D)$
with $\dim_{\C}\overline{X}=4$.

Beginning with a Fano $n$-fold $V$,
Kovalev obtained an admissible pair $(\overline{X},D)$ as follows. 
It is known that there exists a smooth anticanonical divisor $D$ on $V$,
which is a $K3$ surface when $n=3$ and Calabi-Yau $(n-1)$-fold when $n\geqslant 4$.
Let $S$ be a complex $(n-2)$-dimensional submanifold of $D$ such that $S$ represents 
the self-intersection class $D\cdot D$ in $V$.
Then he showed that if $\overline{X}$ is the blow-up of $V$ along $S$, 
the proper transform of $D$ in $\overline{X}$ (which is isomorphic to $D$ and denoted by $D$ again)
is an anticanonical divisor on $\overline{X}$ with the holomorphic normal bundle $N_{D/\overline{X}}$ trivial,
so that $(\overline{X},D)$ is the desired admissible pair.
In \cite{DY14}, we used Fano threefolds $V$ in order to obtain admissible pairs $(\overline{X},D)$
with $\dim_{\C}\overline{X}=3$ in our doubling construction of Calabi-Yau threefolds.
According to Mori-Mukai's classification of Fano threefolds, we gave $59$ topologically distinct
Calabi-Yau threefolds.

In the four-dimensional case, two problems arise in constructing 
Calabi-Yau manifolds by the doubling. 
The first is that we have no complete classification of Fano fourfolds.
The second is that it is not easy to compute the $\widehat{A}$-genus $\widehat{A}(M)$ of the `doubled'
manifold $M$ if we use an arbitrary Fano fourfold. 
Instead of considering all Fano fourfolds, we focus on the \emph{toric} Fano fourfolds which 
are completely classified by Batyrev \cite{Bat98} and Sato \cite{Sato02}.
Also, toric geometry enables us to compute $\Ahat (M)$ systematically.
In fact, using the admissible pair obtained from \emph{any} of the toric Fano fourfolds ($124$ types),
we show that the doubled manifold $M$ has $\Ahat (M)=2$, and hence $M$ is a Calabi-Yau fourfold.
With this construction, we shall give $69$ topologically distinct
Calabi-Yau fourfolds,
whose Euler characteristics $\chi (M)$ range between $936$ and $2688$.

This paper is organized as follows.
Section $2$ is a brief review of $\Spin$-structures. 
In Section $3$ we establish our
gluing construction of Calabi-Yau fourfolds from admissible pairs.
The rest of the paper is devoted to constructing examples
from toric Fano fourfolds.
The reader who is not familiar with analysis can
check Definition $\ref{def:admissible}$ of admissible pairs, go to Section $3.4$ where
the gluing theorems are stated, and then proceed to Section $4$,
skipping Section $2$ and the rest of Section $3$.
In Section $\ref{sec:ToricGeom}$ we outline the quotient construction of toric varieties in
`Geometric Invariant Theory'. We also give a recipe for computing $\widehat{A}(M)$ 
in the proof of Proposition $\ref{prop:CY4}$. Section $\ref{sec:Ex}$ illustrates concrete examples
of our doubling construction of Calabi-Yau fourfolds. Then the last section lists all data
of the resulting Calabi-Yau fourfolds from toric Fano fourfolds.\\

\noindent {\bfseries Acknowledgements.} The second author would like to thank Dr. Craig van Coevering and Dr.
Jinxing Xu for their valuable comments when he was in University of Science and Technology of China.

\section{Geometry of $\Spin$-structures}
Here we shall recall some basic facts about $\Spin$-structures on oriented $8$-manifolds.
For more details, see \cite{J00}, Chapter $10$.

We begin with the definition of $\Spin$-structures on oriented vector spaces of dimension $8$.
\begin{definition}\rm
Let $V$ be an oriented real vector space of dimension $8$.
Let $\{\bm{\theta}^1,\dots ,\bm{\theta}^8\}$ be an oriented basis of $V$.
Set
\begin{equation*}
\begin{aligned}
\bm{\Phi}_0=&\bm{\theta}^{1234}+\bm{\theta}^{1256}+\bm{\theta}^{1278}+\bm{\theta}^{1357}
-\bm{\theta}^{1368}-\bm{\theta}^{1458}-\bm{\theta}^{1467}\\
&-\bm{\theta}^{2358}-\bm{\theta}^{2367}-\bm{\theta}^{2457}+\bm{\theta}^{2468}
+\bm{\theta}^{3456}+\bm{\theta}^{3478}+\bm{\theta}^{5678},\\
\mathbf{g}_0=&\sum_{i=1}^8\bm{\theta}^i\otimes\bm{\theta}^i,
\end{aligned}
\end{equation*}
where $\bm{\theta}^{ij\dots k}=\bm{\theta}^i\wedge\bm{\theta}^j\wedge\dots\wedge\bm{\theta}^k$.
Define the $\GL_+(V)$-orbit spaces
\begin{align*}
\mathcal{A}(V)&=\Set{a^*\bm{\Phi}_0|a\in\GL_+(V)},\\
\mathcal{M}et(V)&=\Set{a^*\mathbf{g}_0|a\in\GL_+(V)}.
\end{align*}
\end{definition}
We call $\mathcal{A}(V)$ the set of {\it Cayley $4$-forms}
(or the set of {\it $\Spin$-structures}) on $V$.
On the other hand, $\mathcal{M}et(V)$ is the set of positive-definite inner products on $V$,
which is also a homogeneous space isomorphic to $\GL_+(V)/\SO(V)$, where $\SO(V)$ is defined by
\begin{equation*}
\SO(V)=\Set{a\in\GL_+(V)|a^*\mathbf{g}_0=\mathbf{g}_0}.
\end{equation*}

Now the group $\Spin$ is defined as the isotropy of the action of $\GL(V)$ (in place of $\GL_+(V)$)
on $\mathcal{A}(V)$ at $\bm{\Phi}_0$:
\begin{equation*}
\Spin =\Set{a\in\GL(V)|a^*\bm{\Phi}_0=\bm{\Phi}_0}.
\end{equation*}
Then one can show that $\Spin$ is a compact Lie group of dimension $27$ which is a
Lie subgroup of $\SO(V)$ \cite{Hy90}.
Thus we have a natural projection
\begin{equation*}
\xymatrix{\mathcal{A}(V)\cong\GL_+(V)/\Spin\ar@{>>}[r]&\GL_+(V)/\SO(V)\cong\mathcal{M}et(V)},
\end{equation*}
so that each Cayley $4$-form (or $\Spin$-structure) $\bm{\Phi}\in\mathcal{A}(V)$
defines a positive-definite inner product $\mathbf{g}_{\bm{\Phi}}\in\mathcal{M}et(V)$ on $V$.
\begin{definition}\label{def:T(V)}\rm
Let $V$ be an oriented vector space of dimension $8$. 
If $\bm{\Phi}\in\mathcal{A}(V)$, then we have the orthogonal decomposition
\begin{equation}\label{eq:ortho_decomp}
\wedge^4 V^*=T_{\bm{\Phi}}\mathcal{A}(V)\oplus T_{\bm{\Phi}}^\perp\mathcal{A}(V)
\end{equation}
with respect to the induced inner product $\mathbf{g}_{\bm{\Phi}}$.
We define a neighborhood 
$\mathcal{T}(V)$ of $\mathcal{A}(V)$ in $\wedge^4 V^*$ by
\begin{equation*}
\begin{aligned}
\mathcal{T}(V)=\Set{
\bm{\Phi}+\bm{\alpha}|
\bm{\Phi}\in\mathcal{A}(V)\text{ and }
\bm{\alpha}\in T_{\bm{\Phi}}^\perp\mathcal{A}(V)
\text{ with }\norm{\bm{\alpha}}_{{\mathbf{g}}_{\bm{\Phi}}}<\rho}.
\end{aligned}
\end{equation*}
We choose and fix a small constant $\rho$ so that 
any $\bm{\chi}\in\mathcal{T}(V)$ is uniquely written
as $\bm{\chi}=\bm{\Phi}+\bm{\alpha}$ with $\bm{\alpha}\in T_{\bm{\Phi}}^\perp\mathcal{A}(V)$.
Thus we can define the projection
\begin{equation*}
\Theta :\mathcal{T}(V)\longrightarrow\mathcal{A}(V),\qquad \bm{\chi}\longmapsto\bm{\Phi}.
\end{equation*}
\end{definition}
\begin{lemma}[Joyce \cite{J00}, Proposition $10.5.4$]\label{lem:ASD-subspace}
Let $\bm{\Phi}\in\mathcal{A}(V)$ and 
$\wedge^4 V^*=\wedge^4_+ V^*\oplus \wedge^4_- V^*$ be the orthogonal decomposition with 
respect to $\mathbf{g}_{\bm{\Phi}}$, where $\wedge^4_+ V^*$ (resp. $\wedge^4_- V^*$) 
is the set of self-dual (resp. anti-self-dual) $4$-forms on $V$.
Then we have the following inclusion:
\begin{equation*}
\wedge^4_- V^*\subset T_{\bm{\Phi}}\mathcal{A}(V).
\end{equation*}
\end{lemma}
Now we define $\Spin$-structures on oriented $8$-manifolds.
\begin{definition}\rm
Let $M$ be an oriented $8$-manifold.
We define $\mathcal{A}(M)\longrightarrow M$ to be the fiber bundle whose fiber over $x$ is
$\mathcal{A}(T^*_x M)\subset\wedge^4 T^*_x M$. Then
$\Phi\in C^\infty (\wedge^4 T^* M)$ is a {\it Cayley $4$-form}
or a {\it $\Spin$-structure} on $M$ if
$\Phi\in C^\infty (\mathcal{A}(M))$, i.e.,
$\Phi$ is a smooth section of $\mathcal{A}(M)$.
If $\Phi$ is a $\Spin$-structure on $M$, then $\Phi$ induces a Riemannian metric $g_\Phi$
since $\restrict{\Phi}{x}$ for each $x\in M$ induces a positive-definite inner product
$g_{\restrict{\Phi}{x}}$ on $T_x M$. A $\Spin$-structure $\Phi$ on $M$ is said to be
{\it torsion-free} if it is parallel with respect to the induced Riemannian metric $g_\Phi$,
i.e., $\nabla_{g_\Phi}\Phi =0$,
where $\nabla_{g_\Phi}$ is the Levi-Civita connection of $g_\Phi$.
\end{definition}
\begin{definition}\rm
Let $\Phi$ be a $\Spin$-structure on an oriented $8$-manifold $M$. We define $\mathcal{T}(M)$
be the the fiber bundle whose fiber over $x$ is $\mathcal{T}(T^*_x M)\subset\wedge^4 T^*_x M$. 
Then for the constant $\rho$ given in Definition $\ref{def:T(V)}$,
we have the well-defined projection $\Theta :\mathcal{A}(M)\longrightarrow\mathcal{T}(M)$.
Also, we see from Lemma $\ref{lem:ASD-subspace}$ that $\wedge^4_- T^*M\subset T_\Phi\mathcal{A}(M)$
as subbundles of $\wedge^4 T^*M$. 
\end{definition}
\begin{lemma}[Joyce \cite{J00}, Proposition $10.5.9$]\label{lem:epsilons}
Let $\Phi$ be a $\Spin$-structure on $M$.
There exist such that $\epsilon_1,\epsilon_2,\epsilon_3$ independent of $M$ and $\Phi$,
such that the following is true.

If $\eta\in C^\infty (\wedge^4T^* M)$ satisfies $\Norm{\eta}_{C^0}\leqslant\epsilon_1$, then $\Phi +\eta\in \mathcal{T}(M)$.
For this $\eta$, $\Theta (\Phi +\eta )$ is well-defined and expanded as
\begin{equation}\label{eq:expansion}
\Theta (\Phi +\eta )=\Phi +p(\eta )-F(\eta ),
\end{equation}
where $p(\eta )$ is the linear term and $F(\eta )$ is the higher order term in $\eta$,
and for each $x\in M$, $\restrict{p(\eta )}{x}$ is the $T_\Phi\mathcal{A}(V)$-component of $\restrict{\eta}{x}$ 
in the orthogonal decomposition \eqref{eq:ortho_decomp} for $V=T^*_x M$. 
Also, we have the following pointwise estimates for any $\eta, \eta'\in C^\infty (\wedge^4 T^* M)$
with $\norm{\eta},\norm{\eta'}\leqslant\epsilon_1$:
\begin{equation*}
\begin{aligned}
\norm{F(\eta )-F(\eta')}\leqslant& \epsilon_2\norm{\eta -\eta'}(\norm{\eta}+\norm{\eta'}),\\
\norm{\nabla (F(\eta )-F(\eta'))}\leqslant &\epsilon_3 \{
\norm{\eta -\eta'}(\norm{\eta}+\norm{\eta'})\norm{\der\Phi}+\norm{\nabla (\eta -\eta')}(\norm{\eta}+\norm{\eta'})\\
&\phantom{\epsilon_3\{ }+\norm{\eta -\eta'}(\norm{\nabla\eta}+\norm{\nabla\eta'})\} .
\end{aligned}
\end{equation*}
Here all norms are measured by $g_\Phi$.
\end{lemma}

The following result is important in that it relates the holonomy contained in $\Spin$ 
with the $\der$-closedness of the $\Spin$-structure.
\begin{theorem}[Salamon \cite{S89}, Lemma $12.4$]\label{thm:d-closed_Spin}
Let $M$ be an oriented $8$-manifold. Let $\Phi$ be a $\Spin$-structure on $M$ and $g_\Phi$
the induced Riemannian metric on $M$.
Then the following conditions are equivalent.
\renewcommand{\labelenumi}{\theenumi}
\renewcommand{\theenumi}{(\arabic{enumi})}
\begin{enumerate}
\item $\Phi$ is a torsion-free $\Spin$-structure, i.e., $\nabla_{g_\Phi}\Phi =0$.
\item $\der\Phi =0$.
\item The holonomy group $\Hol (g_\Phi )$ of $g_\Phi$ is contained in $\Spin$.
\end{enumerate}
\end{theorem}
Now suppose $\widetilde{\Phi}\in C^\infty(\mathcal{T}(M))$ with $\der\widetilde{\Phi}=0$.
We shall construct such a form $\widetilde{\Phi}$ in Section $\ref{section:T-approx}$.
Then $\Phi =\Theta (\widetilde{\Phi})$ is a $\Spin$-structure on $M$.
If $\eta\in C^\infty(\wedge^4 T^*M)$ with $\Norm{\eta}_{C^0}\leqslant\epsilon_1$,
then $\Theta (\Phi +\eta )$ is expanded as in \eqref{eq:expansion}.
Setting $\phi =\widetilde{\Phi}-\Phi$ and using $\der\widetilde{\Phi}=0$, we have 
\begin{equation*}
\der\Theta (\Phi +\eta )=-\der\phi +\der p(\eta )-\der F(\eta ).
\end{equation*}
Thus the equation $\der\Theta (\Phi +\eta )=0$ for $\Theta (\Phi +\eta )$ to be a torsion-free 
$\Spin$-structure is equivalent to
\begin{equation}\label{eq:torsion-free_1}
\der p (\eta )=\der\phi +\der F(\eta ).
\end{equation}
In particular, we see from Lemma $\ref{lem:ASD-subspace}$ that 
if $\eta\in C^\infty (\wedge^4_-T^*M)$ then $p(\eta )=\eta$,
so that equation \eqref{eq:torsion-free_1} becomes
\begin{equation}\label{eq:torsion-free_2}
\der \eta=\der\phi +\der F(\eta ).
\end{equation}
Joyce proved by using the iteration method and 
$\der C^\infty (\wedge^4_- T^*M)=\der C^\infty (\wedge^4 T^*M)$
that equation \eqref{eq:torsion-free_2} has a solution $\eta\in C^\infty (\wedge^4_-T^*M)$
if $\phi$ is sufficiently small with respect to certain norms (see Theorem $\ref{thm:Spin_existence}$).
\begin{theorem}[Joyce \cite{J00}, Theorem $10.6.1$]\label{thm:A-hat}
Let $(M,g)$ be a compact Riemannian $8$-manifold such that its holonomy group $\Hol(g)$ is contained in $\Spin$.
Then the $\Ahat$-genus $\Ahat (M)$ of $M$ satisfies
\begin{equation}\label{eq:A-hat}
48\Ahat (M)=3\tau (M)-\chi (M),
\end{equation}
where $\tau (M)$ and $\chi(M)$ is the signature and the Euler characteristic of $M$ respectively.
Moreover, if $M$ is simply-connected, then $\Ahat (M)$ is $1,2,3$ or $4$, and the holonomy group of $(M,g)$ is determined as
\begin{equation*}
\Hol (g)=\begin{cases}
\Spin&\text{if }\Ahat (M)=1,\\
\SU (4)&\text{if }\Ahat (M)=2,\\
\Symp (2)&\text{if }\Ahat (M)=3,\\
\Symp (1)\times\Symp (1)&\text{if }\Ahat (M)=4.
\end{cases}
\end{equation*}
\end{theorem}
\section{The gluing procedure}

\subsection{Compact complex manifolds with an anticanonical divisor}\label{section:CMWAD}
We suppose that $\overline{X}$ is a compact complex manifold of dimension $n$,
and $D$ is a smooth irreducible anticanonical divisor on $\overline{X}$.
We recall some results in \cite{D09}, Sections $3.1$--$3.2$, and \cite{DY14}, Sections $3.1$--$3.2$.
\begin{lemma}\label{lem:coords_on_X}
Let $\overline{X}$ and $D$ be as above.
Then there exists a local coordinate system
$\{ U_\alpha ,(z_\alpha^1,\dots ,z_\alpha^{n-1},w_\alpha)\}$ on $\overline{X}$
such that
\renewcommand{\labelenumi}{\theenumi}
\renewcommand{\theenumi}{(\roman{enumi})}
\begin{enumerate}
\item $w_\alpha$ is a local defining function of $D$ on $U_\alpha$,
i.e., $D\cap U_\alpha =\{w_\alpha =0\}$, and
\item the $n$-forms $\displaystyle\Omega_\alpha =
\frac{\der w_\alpha}{w_\alpha}\wedge\der z_\alpha^1\wedge\dots\wedge\der z_\alpha^{n-1}$
on $U_\alpha \setminus D$ together yield a holomorphic volume form $\Omega$ on $X=\overline{X}\setminus D$.
\end{enumerate}
\end{lemma}
Next we shall see that $X=\overline{X}\setminus D$ is a cylindrical manifold whose structure is
induced from the holomorphic normal bundle $N=N_{D/\overline{X}}$ to $D$ in $\overline{X}$,
where the definition of cylindrical manifolds is given as follows.
\begin{definition}\label{def:cyl.mfd}\rm
Let $X$ be a noncompact differentiable manifold of real dimension $r$.
Then $X$ is called a {\it cylindrical manifold} or a {\it manifold
with a cylindrical end} if there exists a diffeomorphism $\pi
:X\setminus X_0\longrightarrow\Sigma\times\R_+=\Set{(p,t)|p\in\Sigma
,0<t<\infty}$ for some compact submanifold $X_0$ of dimension $r$
with boundary $\Sigma=\del X_0$. Also, extending $t$ smoothly to $X$ so that
$t\leqslant 0$ on $X\setminus X_0$, we call $t$ a {\it cylindrical
parameter} on $X$.
\end{definition}
Let $(x_\alpha ,y_\alpha)$ be local coordinates on $V_\alpha =U_\alpha\cap D$,
such that $x_\alpha$ is the restriction of $z_\alpha$ to $V_\alpha$ and
$y_\alpha$ is a coordinate in the fiber direction.
Then one can see easily that $\der x_\alpha^1\wedge\dots\wedge\der x_\alpha^{n-1}$ on
$V_\alpha$ together yield a holomorphic volume form $\Omega_D$, which is also called the
{\it Poincar\'{e} residue} of $\Omega$ along $D$.
Let $\Norm{\cdot}$ be the norm of a Hermitian bundle metric on $N$.
We can define a cylindrical parameter $t$ on $N$
by $t=-\frac{1}{2}\log\Norm{s}^2$ for $s\in N\setminus D$.
Then the local coordinates $(z_\alpha ,w_\alpha )$ on $X$ are asymptotic to
the local coordinates $(x_\alpha ,y_\alpha )$ on $N\setminus D$ in the following sense.

\begin{lemma}\label{lem:tub.nbd.thm}
There exists a diffeomorphism $\varphi$ from a neighborhood $V$ of the
zero section of $N$ containing $t^{-1}(\R_+ )$ to a tubular neighborhood $U$ of $D$
in $X$ such that $\varphi$ can be locally written as
\begin{equation*}
\begin{aligned}
z_\alpha &=x_\alpha + O(\norm{y_\alpha}^2)=x_\alpha +O(e^{-t}),\\
w_\alpha &=y_\alpha + O(\norm{y_\alpha}^2)=y_\alpha +O(e^{-t}),
\end{aligned}
\end{equation*}
where we multiply all $z_\alpha$ and $w_\alpha$ by a single constant
to ensure $t^{-1}(\R_+ )\subset V$ if necessary.
\end{lemma}
Hence $X$ is a cylindrical manifold with the cylindrical parameter $t$
via the diffeomorphism $\varphi$ given in the above lemma.
In particular, when $H^0(\overline{X},\mathcal{O}_{\overline{X}})=0$ and
$N_{D/\overline{X}}$ is trivial, we have a useful coordinate system near $D$.
\begin{lemma}[\cite{DY14}, Lemma $3.4$]\label{lem:existence_w}
Let $(\overline{X},D)$ be as in Lemma $\ref{lem:coords_on_X}$. If
$H^1(\overline{X},\mathcal{O}_{\overline{X}})=0$ and the normal
bundle $N_{D/\overline{X}}$ is holomorphically trivial, then there
exist an open neighborhood $U_D$ of $D$ and a holomorphic function
$w$ on $U_D$ such that $w$ is a local defining function of $D$ on
$U_D$. Also, we may define the cylindrical parameter $t$ with
$t^{-1}(\R_+)\subset U_D$ by writing the fiber coordinate $y$ of
$N_{D/\overline{X}}$ as $y=\exp (-t-\I\theta )$.
\end{lemma}

\subsection{Admissible pairs and asymptotically cylindrical Ricci-flat K\"{a}hler manifolds}

\begin{definition}\rm
Let $X$ be a cylindrical manifold such that
$\pi :X\setminus X_0\longrightarrow\Sigma\times\R_+=\{(p,t)\}$
is a corresponding diffeomorphism.
If $g_\Sigma$ is a Riemannian metric on $\Sigma$, then it defines a cylindrical metric
$g_{\rm cyl}=g_\Sigma +\der t^2$ on $\Sigma\times\R_+$.
Then a complete Riemannian metric $g$ on $X$ is said to be {\it asymptotically cylindrical}
({\it to }$(\Sigma\times\R_+,g_{\rm cyl})$) if $g$ satisfies 
for some cylindrical metric $g_{\rm cyl}=g_\Sigma +\der t^2$
\begin{equation*}
\norm{\nabla_{g_{\rm cyl}}^j(g-g_{\rm cyl})}_{g_{\rm cyl}}\longrightarrow 0
\qquad\text{as }t\longrightarrow\infty\qquad\text{for all }j\geqslant 0,
\end{equation*}
where we regarded $g_{\rm cyl}$ as a Riemannian metric on $X\setminus X_0$
via the diffeomorphism $\pi$.
Also, we call $(X,g)$ an {\it asymptotically cylindrical manifold} and
$(\Sigma\times\R_+,g_{\rm cyl})$ the {\it asymptotic model} of $(X,g)$.
\end{definition}
\begin{definition}\label{def:admissible}\rm
Let $\overline{X}$ be a complex manifold and $D$ a divisor on $\overline{X}$.
Then $(\overline{X},D)$ is said to be an {\it admissible pair} if the following conditions hold:
\begin{enumerate}
\item[(a)] $\overline{X}$ is a compact K\"ahler manifold.
\item[(b)] $D$ is a smooth anticanonical divisor on $\overline{X}$.
\item[(c)] the normal bundle $N_{D / \overline{X}}$ is trivial.
\item[(d)] $\overline{X}$ and $\overline{X}\setminus D$ are simply-connected.
\end{enumerate}
\end{definition}
From the above conditions, we see that Lemmas $\ref{lem:coords_on_X}$ and
$\ref{lem:existence_w}$ apply to admissible pairs.
Also, from conditions (a) and (b), we see that $D$ is a compact K\"{a}hler manifold
with trivial canonical bundle. Thus $D$ admits a Ricci-flat K\"{a}hler metric.

\begin{theorem}[Tian-Yau \cite{TY90}, Kovalev \cite{K03}, Hein \cite{Hn10}]\label{thm:TYKH}
Let $(\overline{X},\omega' )$ be a compact K\"{a}hler manifold
and $n=\dim_\C\overline{X}$. 
If $(\overline{X},D)$ is an admissible pair, then the following is true.

It follows from Lemmas $\ref{lem:coords_on_X}$ and
$\ref{lem:existence_w}$, there exist a local coordinate system
$(U_{D,\alpha} ,(z_\alpha^1,\dots ,z_\alpha^{n-1},w))$ on a
neighborhood $U_D=\cup_\alpha U_{D,\alpha}$ of $D$ and a holomorphic
volume form $\Omega$ on $\overline{X} \setminus D$ such that
\begin{equation}\label{eq:TYKH_Omega}
\Omega =\frac{\der w}{w}\wedge\der z_\alpha^1\wedge\dots\wedge
\der z_\alpha^{n-1}\qquad\text{on }U_{D,\alpha}\setminus D.
\end{equation}
Let $\kappa_D$ be the unique Ricci-flat K\"{a}hler form on $D$ in the K\"{a}hler class
$[\restrict{\omega'}{D}]$.
Also let $(x_\alpha ,y)$ be local coordinates of $N_{D/\overline{X}}\setminus D$
as in Section $\ref{section:CMWAD}$ and write $y$ as $y=\exp (-t-\I\theta )$.
Now define a holomorphic volume form $\Omega_{\rm cyl}$
and a cylindrical Ricci-flat K\"{a}hler form $\omega_{\rm cyl}$ on $N_{D/\overline{X}}\setminus D$ by
\begin{equation}\label{eq:TYKH_CYcyl}
\begin{aligned}
\Omega_{\rm cyl}&=\frac{\der y}{y}\wedge\der x_\alpha^1\wedge\dots\wedge\der x_\alpha^{n-1}
=(\der t +\I\der\theta)\wedge\Omega_D,\\
\omega_{\rm cyl}&=\kappa_D+\frac{\I}{2}\frac{\der y\wedge\der\overline{y}}{\norm{y}^2}
=\kappa_D+\der t\wedge\der\theta .
\end{aligned}
\end{equation}
Then there exist a holomorphic volume form $\Omega$ and an asymptotically cylindrical Ricci-flat K\"{a}hler form $\omega$ on
$X=\overline{X}\setminus D$ such that
\begin{equation*}
\Omega -\Omega_{\rm cyl}=\der\zeta,\qquad\omega -\omega_{\rm cyl}=\der\xi
\end{equation*}
for some $\zeta$ and $\xi$ with
\begin{equation*}
\norm{\nabla_{g_{\rm cyl}}^j\zeta}_{g_{\rm cyl}}=O(e^{-\beta t}),\qquad
\norm{\nabla_{g_{\rm cyl}}^j\xi}_{g_{\rm cyl}}=O(e^{-\beta t})
\qquad\text{for all }j\geqslant 0\text{ and }\beta\in (0,\min\set{1/2,\sqrt{\lambda_1}}),
\end{equation*}
where $\lambda_1$ is the first eigenvalue of the Laplacian $\Delta_{g_D+\der\theta^2}$
acting on $D\times S^1$ with $g_D$ the metric associated with $\kappa_D$.
\end{theorem}
A pair $(\Omega ,\omega )$ consisting of a holomorphic volume form $\Omega$ and a Ricci-flat K\"{a}hler form $\omega$
on an $n$-dimensional K\"{a}hler manifold
normalized so that
\begin{equation*}
\frac{\omega^n}{n!}=\frac{(\I)^{n^2}}{2^n}\Omega\wedge\overline{\Omega}\;(=\text{the volume form})
\end{equation*}
is called a {\it Calabi-Yau structure}.
The above theorem states that there exists a Calabi-Yau structure $(\Omega ,\omega)$
on $X$ asymptotic to a cylindrical Calabi-Yau structure
$(\Omega_{\rm cyl},\omega_{\rm cyl})$ on $N_{D/\overline{X}}\setminus D$
if we multiply $\Omega$ by some constant.

\subsection{Gluing admissible pairs}
In this subsection we will only consider admissible pairs $(\overline{X},D)$ with
$\dim_\C\overline{X}=4$. Also, we will denote $N=N_{D/\overline{X}}$ and $X=\overline{X}\setminus D$.

\subsubsection{The gluing condition}\label{sec:gluing_cond}

Let $(\overline{X},\omega')$ be a four-dimensional compact K\"{a}hler manifold
and $(\overline{X},D)$ be an admissible pair.
We first define a natural torsion-free $\Spin$-structure on $X$.

It follows from Theorem $\ref{thm:TYKH}$ that there exists a
Calabi-Yau structure $(\Omega ,\omega)$ on $X$ asymptotic to a cylindrical
Calabi-Yau structure $(\Omega_{\rm cyl},\omega_{\rm cyl})$
on $N\setminus D$, which are written as in \eqref{eq:TYKH_Omega}
and \eqref{eq:TYKH_CYcyl}.
We define a $\Spin$-structure $\Phi$ on $X$ by
\begin{equation}\label{eq:Phi}
\Phi =\frac{1}{2}\omega\wedge\omega +\Real\Omega .
\end{equation}
Similarly, we define a $\Spin$-structure $\Phi_{\rm cyl}$
on $N\setminus D$ by
\begin{equation}\label{eq:Phi_cyl}
\Phi_{\rm cyl} =\frac{1}{2}\omega_{\rm cyl}\wedge\omega_{\rm cyl} +\Real\Omega_{\rm cyl}.
\end{equation}
Then we see easily from Theorem $\ref{thm:TYKH}$ and
equations \eqref{eq:Phi} and \eqref{eq:Phi_cyl} that
\begin{equation}\label{eq:difference}
\begin{aligned}
\Phi -\Phi_{\rm cyl}&=\frac{1}{2}\der\xi\wedge (\Omega +\Omega_{\rm cyl}) +\Real\der\zeta
=\der\eta ,\\
\text{where}\quad\eta &=\frac{1}{2}\xi\wedge (\Omega +\Omega_{\rm cyl}) +\Real\zeta .
\end{aligned}
\end{equation}
Thus $\Phi$ and $\Phi_{\rm cyl}$ are both torsion-free $\Spin$-structures,
and $(X, \Phi )$ is asymptotic to
$(N\setminus D,\Phi_{\rm cyl})$.
Note that the cylindrical end of $X$ is diffeomorphic to
$N\setminus D\simeq D\times S^1\times\R_+
=\{(x_\alpha ,\theta ,t)\}$.

Next we consider the condition under which we can glue together
$X_1$ and $X_2$ obtained from admissible pairs
$(\overline{X}_1,D_1)$ and $(\overline{X}_2,D_2)$. For gluing $X_1$
and $X_2$ to obtain a manifold with a $\Spin$-structure with small torsion, 
we would like $(X_1,\Phi_1 )$ and $(X_2,\Phi_2)$ to have the same asymptotic model. 
Thus we put the following
\begin{description}
\item[\it Gluing condition] There exists a diffeomorphism
$F: D_1\times S^1\longrightarrow D_2\times S^1$
between the cross-sections of the cylindrical ends such that
\begin{equation}\label{eq:gluing_condition}
F_T^*\Phi_{2,\rm cyl} =\Phi_{1,\rm cyl}\qquad\text{for all }T>0,
\end{equation}
where $F_T:D_1\times S^1\times (0,2T)\longrightarrow D_2\times S^1\times (0,2T)$
is defined by
\begin{equation*}
F_T(x_1,\theta_1, t)=(F( x_1,\theta_1),2T-t)\qquad\text{for }
(x_1,\theta_1, t)\in D_1\times S^1\times (0,2T) .
\end{equation*}
\end{description}
\begin{lemma}\label{lem:gluing_condition}
Suppose that
there exists an isomorphism $f:D_1\longrightarrow D_2$ such that $f^*\kappa_{D_2}=\kappa_{D_1}$.
If we define a diffeomorphism $F$ between the cross-sections of the cylindrical ends by
\begin{equation*}
F:D_1\times S^1\longrightarrow D_2\times S^1,\qquad 
(x_1,\theta_1 )\longmapsto (x_2,\theta_2 )=(f(x_1),-\theta_1)
\end{equation*}
Then the gluing condition \eqref{eq:gluing_condition} holds,
where we change the sign of $\Omega_{2,\rm cyl}$
(and also the sign of $\Omega_2$ correspondingly).
\end{lemma}
\begin{proof}
It follows by a straightforward calculation
using \eqref{eq:TYKH_CYcyl} and \eqref{eq:Phi_cyl}.
\end{proof}

\subsubsection{$\Spin$-structures with small torsion}\label{section:T-approx}
Now we shall glue $X_1$ and $X_2$ under the gluing condition \eqref{eq:gluing_condition}.
Let $\rho :\R\longrightarrow [0,1]$ denote a cut-off function
\begin{equation*}
\rho (x)=
\begin{cases}
1&\text{if }x\leqslant 0,\\
0&\text{if }x\geqslant 1,
\end{cases}
\end{equation*}
and define $\rho_T :\R\longrightarrow [0,1]$ by
\begin{equation*}
\rho_T (x)=\rho (x-T+1)=
\begin{cases}
1&\text{if }x\leqslant T-1,\\
0&\text{if }x\geqslant T.
\end{cases}
\end{equation*}
Setting an approximating Calabi-Yau structure $(\Omega_{i,T}, \omega_{i,T})$ by
\begin{equation*}
\Omega_{i,T}=
\begin{cases}
\Omega_i -\der (1-\rho_{T-1})\zeta_i &\text{on }\{ t_i\leqslant T-1\} ,\\
\Omega_{i,\rm cyl}+\der\rho_{T-1}\zeta_i &\text{on }\{ t_i\geqslant T-2\}
\end{cases}
\end{equation*}
and similarly
\begin{equation*}
\omega_{i,T}=
\begin{cases}
\omega_i -\der (1-\rho_{T-1})\xi_i &\text{on }\{ t_i\leqslant T-1\} ,\\
\omega_{i,\rm cyl}+\der\rho_{T-1}\xi_i &\text{on }\{ t_i\geqslant T-2\},
\end{cases}
\end{equation*}
we can define a $\der$-closed $4$-form
$\widetilde{\Phi}_{i,T}$ on each $X_i$ by
\begin{equation*}
\widetilde{\Phi}_{i,T}=\frac{1}{2}\omega_{i,T}\wedge\omega_{i,T} +\Real\Omega_T .
\end{equation*}
Note that $\widetilde{\Phi}_{i,T}$ satisfies
\begin{equation*}
\widetilde{\Phi}_{i,T}=
\begin{cases}
\Phi_i&\text{on }\{ t_i<T-2\} ,\\
\Phi_{i,\rm cyl}&\text{on }\{ t_i>T-1\}
\end{cases}
\end{equation*}
and that
\begin{equation}\label{eq:Phi_T-Phi_cyl}
\norm{\widetilde{\Phi}_{i,T} -\Phi_{i,\rm cyl}}_{g_{\Phi_{i,\rm cyl}}}=O(e^{-\beta T})
\qquad\text{for all }\beta\in (0,\min\set{1/2,\sqrt{\lambda_1}}).
\end{equation}

Let $X_{1,T}=\{ t_1<T+1\}\subset X_1$ and $X_{2,T}=\{ t_2<T+1\}\subset X_2$.
We glue $X_{1,T}$ and $X_{2,T}$
along $D_1\times\{ T-1<t_1<T+1\}\subset X_{1,T}$
and $D_2\times\{ T-1<t_2<T+1\}\subset X_{2,T}$
to construct a compact $8$-manifold $M_T$ using the gluing map $F_T$
(more precisely, $\widetilde{F}_T=\varphi_2\circ F_T\circ\varphi_1^{-1}$,
where $\varphi_1$ and $\varphi_2$
are the diffeomorphisms given in Lemma $\ref{lem:tub.nbd.thm}$).
Also, we can glue together $\widetilde{\Phi}_{1,T}$ and $\widetilde{\Phi}_{2,T}$ to obtain a 
$\der$-closed $4$-form $\widetilde{\Phi}_T$ on $M_T$ by Lemma $\ref{lem:gluing_condition}$.
There exists a positive constant $T_*$ such that $\widetilde{\Phi}_T\in C^\infty (\mathcal{T}(M_T))$
for any $T$ with $T>T_*$. 
This $\widetilde{\Phi}_T$ is what was discussed right after Theorem $\ref{thm:d-closed_Spin}$, from which
we can define a $\Spin$-structure $\Phi_T$ {\it with small torsion} by
\begin{equation*}
\Phi_T =\Theta (\widetilde{\Phi}_T).
\end{equation*}
Now let 
\begin{equation}\label{eq:phi_T}
\phi_T = \widetilde{\Phi}_T -\Phi_T. 
\end{equation}
Then $\der\phi_T +\der\Phi_T=0$.
\begin{proposition}\label{prop:estimates}
Let $T>T_*$.
Then there exist constants $A_{p,k,\beta}$
independent of $T$ such that for $\beta\in (0,\min\set{1/2,\sqrt{\lambda_1}})$ we have
\begin{equation*}
\Norm{\phi_T}_{L^p_k}\leqslant A_{p,k,\beta}\, e^{-\beta T},
\end{equation*}
where all norms are measured using $g_{\Phi_T}$.
\end{proposition}
\begin{proof}
These estimates follow in a straightforward way from Theorem $\ref{thm:TYKH}$
and equation \eqref{eq:difference}
by arguments similar to those in \cite{D09}, Section 3.5.
\end{proof}

\subsection{Gluing theorems}

\begin{theorem}\label{thm:main}
Let $(\overline{X}_1,\omega'_1)$ and $(\overline{X}_2,\omega'_2)$ be
compact K\"{a}hler manifolds with $\dim_\C\overline{X}_i=4$
such that $(\overline{X}_1,D_1)$ and $(\overline{X}_2,D_2)$ are admissible pairs.
Suppose there exists an isomorphism $f: D_1\longrightarrow D_2$ such that
$f^*\kappa_2=\kappa_1$,
where $\kappa_i$ is the unique Ricci-flat K\"{a}hler form on $D_i$ in the K\"{a}hler class
$[\restrict{\omega'_i}{D_i}]$.
Then we can glue toghether $X_1$ and $X_2$ along their cylindrical ends to obtain
a compact simply-connected $8$-manifold $M$. The manifold $M$ admits a Riemannian metric with holonomy 
contained in $\Spin$. Moreover, if $\Ahat (M)=2$, then $M$ is a Calabi-Yau fourfold, i.e., 
$M$ admits a Ricci-flat K\"{a}hler metric with holonomy $\SU (4)$.
\end{theorem}
\begin{corollary}\label{cor:doubling}
Let $(\overline{X},D)$ be an admissible pair with $\dim_\C\overline{X}=4$.
Then we can glue two copies of $X$ along their cylindrical ends to obtain a compact simply-connected 
$8$-manifold $M$. Then $M$ admits a Riemannian metric with holonomy contained in $\Spin$.
If $\Ahat (M)=2$, then the manifold $M$ is a Calabi-Yau fourfold.
\end{corollary}
\begin{proof}[Proof of Theorem $\ref{thm:main}$]
The assertion for $\Ahat(M)=2$ in Theorem $\ref{thm:main}$ follows directly from Theorem $\ref{thm:A-hat}$.
Thus it remains to prove the existence of a torsion-free $\Spin$-structure on $M_T$
for sufficiently large $T$. 
This is a consequence of the following
\begin{theorem}[Joyce \cite{J00}, Theorem $13.6.1$]\label{thm:Spin_existence}
Let $\lambda ,\mu ,\nu$ be positive constants. Then there exists a positive constant
$\epsilon_*$ such that whenever $0<\epsilon<\epsilon_*$, the following is true.\\
Let $M$ be a compact $8$-manifold and $\Phi$ a $\Spin$-structure on $M$. 
Suppose $\phi$ is a smooth $4$-form on $M$ with $\der\Phi +\der\phi=0$,
and
\begin{enumerate}
\item $\Norm{\phi}_{L^2}\leqslant\lambda\epsilon^{13/3}$
and $\Norm{\der\phi}_{L^{10}}\leqslant\lambda\epsilon^{7/5}$,
\item the injectivity radius $\delta (g)$ satisfies $\delta (g)\geqslant\mu\epsilon$, and
\item the Riemann curvature $R(g)$ satisfies $\Norm{R(g)}_{C^0}\leqslant\nu\epsilon^{-2}$.
\end{enumerate}
Let $\epsilon_1$ be as in Lemma $\ref{lem:epsilons}$.
Then there exists $\eta\in C^\infty (\wedge^4T^*_-M)$ with $\Norm{\eta}_{C^0}<\epsilon_1$
such that $\der\Theta (\Phi +\eta )=0$. Hence the manifold $M$ admits a torsion-free
$\Spin$-structure $\Theta (\Phi +\eta)$.
\end{theorem}
Since $X_1$ and $X_2$ are cylindrical, the injectivity radius and Riemann curvature of
$M_T$ are uniformly bounded from below and above respectively,
conditions (2) and (3) hold for sufficiently large $T$.

For condition (1), we set $\phi =\phi_T$ by equation \eqref{eq:phi_T} for $T>T_*$.
Choosing $\gamma$ so that $0<\gamma <\frac{3}{13}\min\set{1/2,\sqrt{\lambda_1}}$
and letting $\epsilon =e^{-\gamma T}$,
we see from Proposition $\ref{prop:estimates}$ that condition (1) holds for some $\lambda$.
Thus we can apply Theorem $\ref{thm:Spin_existence}$ to prove that $\Phi_T$
can be deformed into a torsion-free $\Spin$-structure for sufficiently large $T$.
This completes the proof of Theorem $\ref{thm:main}$.
\end{proof}

\section{Some results from toric geometry}\label{sec:ToricGeom}
In this section we give a quick review of some related results from toric geometry.
A good reference for the contents of this section is \cite{CLS}.

\subsection{GIT construction of toric varieties}\label{subsec:GIT}
Let $\bM$ be a lattice of rank $n$, $\bN=\Hom (\bM,\Z)$ the $\Z$-dual of $\bM$.
Let $\bM_{\R}$ (resp. $\bN_{\R}$) denote the $\R$-vector space $\bM\otimes_{\Z}\R$ (resp. $\bN\otimes_{\Z}\R$).
Let $\D$ be a fan in $\bN_{\R}$, and $\P_{\D}$ the associated toric variety. Let $\D(1)$ denote the set of
the $1$-dimensional cones of $\D$. To begin with, we shall construct $\P_{\D}$ as a GIT quotient
\[
\P_{\D}\cong (\C^r\setminus Z(\D))\quot G
\]
for an appropriate reductive group $G$, an affine space $\C^r$ and an exceptional set $Z(\D)\subseteq \C^r$ with
$r=\norm{\D(1)}$. Throughout this section, we only consider $n$-dimensional complete
 toric varieties with no torus factors.
For more details, see \cite{Cox95}, \cite{CLS}, Chapter $5$, and \cite{Dolga02}, Chapter $12$. 

Let $\bbT_{\bN}=\bN\otimes_{\Z}\C^*$ be the algebraic torus acting on $\P_{\D}$. The \emph{Orbit-Cone Correspondence} gives
a bijective correspondence between each $\rho \in \D(1)$ and an irreducible $\bbT_{\bN}$-invariant Weil divisor $D_{\rho}$
on $\P_{\D}$. It is well-known that the $\bbT_{\bN}$-invariant Weil divisors on $\P_{\D}$ form a free abelian group, which is
denoted by $\Z^{\D(1)}$. Let $\Div_\bbT (\P_{\D})$ denote the set of all $\bbT_{\bN}$-invariant Cartier divisors on 
$\P_{\D}$. Then $\Div_\bbT (\P_{\D})$ is a subgroup of $\Z^{\D(1)}$. Each $\rho\in \D(1)$ is given by the minimal generator
$u_{\rho}\in \rho\cap \bN$. Recall that $m\in \bM$ gives a character $\chi^m: \bbT_{\bN}\longrightarrow \C^*$ which is a rational function on
$\P_{\D}$, and its divisor is given by $\divisor(\chi^m)=\sum_{\rho}\braket{m,u_{\rho}} D_{\rho}$. In particular, any
divisor $D\in \Z^{\D(1)}$ has the form $D=\sum_{\rho}a_{\rho}D_{\rho}$. Let $\left[\sum_{\rho}a_{\rho}D_{\rho}\right]$ denote its divisor
class in the Chow group $A_{n-1}(\P_{\D})$. Then we have a commutative diagram
\begin{equation}\label{diagram:Ful}
\begin{split}
\xymatrix{
 0\sp \ar@{>}[r] &\sp \bM  \ar@{=}[d] \sp \ar@{>}[r] & \sp \Div_\bbT (\P_{\D}) \ar@{>}[d] \sp \ar@{>}[r] 
 & \sp \Pic(\P_{\D}) \ar@{>}[d] \sp \ar@{>}[r] &\sp  0\\
 0\sp \ar@{>}[r] &\sp \bM \sp \ar@{>}[r]^-{ f} &\sp\sp \Z^{\D(1)} \sp\sp \ar@{>}[r]^-{g} &\sp A_{n-1}(\P_{\D})\sp \ar@{>}[r] &\sp 0
  }
  \end{split}
\end{equation}
by \cite{Ful93}, p. $63$, where maps $f$ and $g$ are defined by
\[
f: \bM \longrightarrow \Z^{\D(1)}, \qquad m\longmapsto D_m=\sum_{\rho\in \D(1)}\braket{m, u_{\rho}} D_{\rho}
\]
and
\[
g: \Z^{\D(1)} \longrightarrow A_{n-1}(\P_{\D}), \qquad a=(a_{\rho})\longmapsto  \left[\sum_{\rho}a_{\rho}D_{\rho}\right]
\]
respectively. Note that the rows are exact and the vertical arrows are inclusion maps in \eqref{diagram:Ful}.
Since $\Hom_{\Z}(\;\cdot\; ,\C^*)$ is left-exact and $\C^*$ is divisible, the bottom row in \eqref{diagram:Ful} induces an exact
sequence of affine algebraic groups
\begin{equation}\label{seq:alg-gr}
0 \longrightarrow G \longrightarrow  (\C^*)^{\Delta(1)} \longrightarrow  \bbT_{\bN} \longrightarrow  1,
\end{equation}
where $G=\Hom_{\Z}(A_{n-1}(\P_{\D}),\C^*)$ and $(\C^*)^{\D(1)}=\Hom_{\Z}(\Z^{\D(1)},\C^*)$.
We note that $A_{n-1}(\P_{\D})$ is the character group of $G$ (see \cite{CLS}, p. $206$).
Introducing a variable $x_{\rho}$ for each $\rho \in \D(1)$, we define the total coordinate ring $S(\D)$ of $\P_{\D}$ by 
\[
S(\D)=\C[x_{\rho}\mid \rho \in \D(1)].
\]
Then $\Spec(S(\D))$ is an affine space $\C^{\D(1)}$.

A subset $C\subseteq \D(1)$ is said to be a \emph{primitive collection} if the following conditions hold:
\begin{enumerate}
\item $C\nsubseteq \sigma(1)$ for all $\sigma\in \D$.
\item For every proper subset $C'\subsetneq C$, there is $\sigma \in \D$ such that $C'\subseteq \sigma(1)$.
\end{enumerate}
Let $\PC(\D)$ denote the set of all primitive collections of $\D$. For a given primitive collection 
$C\in \PC(\D)$, we consider the subspace of $\C^{\D(1)}$
\[
\C^{\D(1)}\supseteq {\mathbf V}(x_{\rho} \mid \rho\in C) = \Set{ (x_{\rho})\in \C^{\D(1)} | x_{\rho}=0
\text{\;\;if\;\;} \rho\in C }.
\]
Then the union of ${\mathbf V}(x_{\rho} \mid \rho \in C)$ over all primitive collections gives the variety
\[
Z(\D)=\bigcup_{C \in \PC(\D)}{\mathbf V}(x_{\rho} \mid \rho \in C)
\]
in $\C^{\D(1)}$.
Observe that $(\C^*)^{\D(1)}$ acts diagonally on $\C^{\D(1)}$. This induces an action on $\C^{\D(1)}\setminus Z(\D)$ and
hence $G\subseteq (\C^*)^{\D(1)}$ acts on $\C^{\D(1)}\setminus Z(\D)$. In \cite{Cox95}, Cox gave the quotient construction
of toric varieties, i.e., $\P_{\D}$ is an (almost) geometric quotient for the action of $G$ on $\C^{\D(1)}\setminus Z(\D)$.
Thus we have
\begin{equation}\label{iso:quot}
\P_{\D}\cong (\C^{\D(1)}\setminus Z(\D))\quot G.
\end{equation}
Then \eqref{seq:alg-gr} and \eqref{iso:quot} induce a commutative diagram
\begin{equation}\label{diagram:toric}
\begin{split}
\xymatrix{
 \bbT_{\bN}\ar[d]\ar[r]^-\cong  & (\C^*)^{\D(1)}/G\ar[d]\\
 \P_{\D}\ar[r]^-\cong & (\C^{\D(1)}\setminus Z(\D))\quot G
}
\end{split}
\end{equation}
where the vertical arrows are inclusion maps. Diagram \eqref{diagram:toric} is consistent with the usual definition of toric varieties:
a toric variety $X$ is a normal irreducible algebraic variety containing a torus 
$\bbT_{\bN}$ as a Zariski open subset such that the action
of $\bbT_{\bN}\cong (\C^*)^n$ on itself extends to an algebraic action of $\bbT_{\bN}$ on $X$.

\subsection{The grading of $S(\D)$}\label{subsec:grading}
The homogeneous coordinate ring $S(\D)$ of an $n$-dimensional toric variety $\P_{\D}$ has a natural grading by the Chow group
$A_{n-1}(\P_{\D})$. In the bottom exact sequence of \eqref{diagram:Ful}, $a=(a_{\rho})\in \Z^{\D(1)}$ maps to the divisor class
$\left[\sum_{\rho}a_{\rho}D_{\rho}\right]\in A_{n-1}(\P_{\D})$. For a given monomial $x^a=\prod_{\rho}x_{\rho}^{a_{\rho}}\in S(\D)$, the degree of $x^a$ is 
\[
\deg (x^a)=\left[\sum_{\rho}a_{\rho}D_{\rho}\right]\in A_{n-1}(\P_{\D}).
\]
We denote $S(\D)_{\alpha}$ the corresponding graded piece of $S(\D)$ for $\alpha \in A_{n-1}(\P_{\D})$. Since $A_{n-1}(\P_{\D})$
is the character group of $G=\Hom_{\Z}(A_{n-1}(\P_{\D}),\C^*)$, $\alpha \in A_{n-1}(\P_{\D})$ gives the character 
$\chi^{\alpha}:G\longrightarrow \C^*$. Then the action of $G$ on $\C^{\D(1)}$ induces an action of $G$ on $S(\D)$ through the 
character $\chi^{\alpha}$. A polynomial $f\in S(\D)_{\alpha}$ is said to be \emph{$G$-homogeneous} of degree $\alpha$.

\subsection{Cohomology of toric complete intersections}
Hereafter we assume $\P_{\D}$ is a \emph{smooth} complete toric variety. Then it is easy to compute the cohomology
ring of $\P_{\D}$ as follows.

Let us fix an order of the rays $\rho_1,\dots,\rho_r$ in $\D(1)$. As in Section $\ref{subsec:GIT}$, for a given ray $\rho_i\in \D(1)$,
let $u_i$ denote the minimal generator of $\rho_i$ and $x_i$ the corresponding variable. The \emph{Stanley-Reisner ideal}
of $\D$ is the squarefree monomial ideal
\[
\mathscr{I}_{\D}=\braket{x_{i_1}\cdots x_{i_{\ell}}\mid i_j \text{ are distinct and } \Cone(u_{i_1},\dots, u_{i_{\ell}}) \notin \D}
\] 
in the ring $\Z[x_1,\dots,x_r]$. Note that $\PC(\D)$ generates $\mathscr{I}_{\D}$.
On the other hand, we consider the ideal $\mathcal{J}_{\D}$ generated by linear combinations
\[
\sum_{i=1}^r \braket{m, u_i}x_i,
\]
where $m$ runs over some basis of $\bM$. Since the ideal $\mathscr{I}_{\D}+\mathcal{J}_{\D}$ is homogeneous in 
$\Z[x_1,\dots,x_r]$
with respect to the grading defined in Section $\ref{subsec:grading}$, the quotient 
$R_{\D}=\Z[x_1,\dots,x_r]/(\mathscr{I}_{\D}+\mathcal{J}_{\D})$ 
is a graded ring. If we coonsider the ring structure determined by the cup product on
\begin{equation*}
H^{\bullet}(\P_{\D},\Z)=\bigoplus_{k=0}^{2n}H^k(\P_{\D},\Z), \qquad n=\dim_{\C}\P_{\D},
\end{equation*}
then we can show that $H^{\bullet}(\P_{\D},\Z)$ is isomorphic to $R_{\D}$.
\begin{proposition}[Jurkiewicz-Danilov]\label{prop:cohomology}
Let $\P_{\D}$ be a smooth complete toric variety. Then the map
\[
R_{\D} \longrightarrow H^{\bullet}(\P_{\D},\Z), \qquad x_i \longmapsto [D_{\rho_i}]
\]
induces a ring isomorphism $R_{\D}\cong H^{\bullet}(\P_{\D},\Z)$.
\end{proposition}
Now we can compute the Chern classes of smooth complete toric varieties using the following results.
\begin{proposition}[\cite{CLS}, Proposition $13.1.2$]\label{prop:chern}
Let $\P_{\D}$ be a smooth complete toric variety. Then we have
\begin{itemize}
\item[(i)] $c(\P_{\D})=\prod_{\rho \in \D(1)}\left(1+[D_{\rho}]\right)$.
\item[(ii)] $c_1(\P_{\D})=\left[\sum_{\rho\in \Delta(1)}D_{\rho}\right]=[-K_{\P_{\D}}]$,
\end{itemize}
where $-K_{\P_{\D}}=\sum_{\rho\in \Delta(1)}D_{\rho}$ is a torus-invariant anticanonical divisor on $\P_{\D}$.
\end{proposition}
We use the following {\it Noether's formula} in order to compute the cohomology of complete intersections in $\P_{\D}$ 
at a later stage in our argument (Section $\ref{sec:Ex}$).
\begin{theorem}[Noether's formula]\label{th:Noether}
Let $S$ be a complex $2$-dimensional compact manifold and $h^{p,q}$ with $p,q\in\set{0,1,2}$ be the
Hodge numbers of $S$. Then we have
\begin{align}\label{eq:Noether}
\begin{split}
h^{0,0}-h^{0,1}+h^{0,2}&=\frac{1}{12}\int_S c_1(S)^2+c_2(S), \\
h^{1,0}-h^{1,1}+h^{1,2}&=\frac{1}{6}\int_S c_1(S)^2-5c_2(S), \\
h^{2,0}-h^{2,1}+h^{2,2}&=\frac{1}{12}\int_S c_1(S)^2+c_2(S).
\end{split}
\end{align}
\end{theorem}

\subsection{How to find $\widehat{A}(M)$}
According to the classification results of toric Fano fourfolds \cite{Bat98}, \cite{Sato02}, we can construct
a Calabi-Yau fourfold by the doubling construction from any of the $124$ types of toric Fano fourfolds.
\begin{proposition}\label{prop:CY4}
Let $\P_{\Delta}$ be a toric Fano fourfold and let $(\overline{X},D)$ be the corresponding admissible pair 
of Fano type given in \cite{DY14}, Theorem $5.1$.
Then the simply-connected compact $8$-manifold $M$ 
obtained from two copies of $(\overline{X},D)$ by Corollary $\ref{cor:doubling}$
satisfies $\widehat{A}(M)=2$. In particular, $M$ admits a Ricci-flat K\"ahler metric with holonomy group
$\mathrm{SU}(4)$.
\end{proposition}
The proof of Proposition $\ref{prop:CY4}$ is based on the computation of $\widehat{A}(M)$ one by one.
Here is a procedure for calculating $\widehat{A}(M)$ by toric geometrical technique which we have already 
explained above.
\begin{enumerate}
\item[(a)] According to the classification result of toric Fano fourfolds due to Batyrev \cite{Bat98} and Sato
\cite{Sato02}, there are $124$ types of toric Fano fourfolds. A database of classifications of
smooth toric Fano varieties is available at \cite{graded_ring}. For a fixed $4$-dimensional complete fan $\D$
(which is also called a \emph{Fano polytope}), find the primitive collection $\PC (\D)$ from the lists
in \cite{Bat98}, \cite{Sato02}. Then Proposition $\ref{prop:cohomology}$ implies that
\begin{equation}\label{eq:cohomology}
H^{\bullet}(\P_{\D},\Z)=\Z[x_1,\dots,x_r]/(\mathscr{I}_{\D}+\mathcal{J}_{\D}).
\end{equation}
It is easy to compute the right hand side of \eqref{eq:cohomology} by $\PC(\D)$ and $\D(1)$.
\item[(b)] Suppose that the total coordinate ring $S(\D)$ of $\P_{\D}$ is given by
\[
S(\D)=\C[x_{\rho}\mid\rho \in \D(1)]
\]
with $\deg(x_{\rho})=a_{\rho}$. 
Recall that the anticanonical divisor $-K_{\P_{\D}}$ on $\P_{\D}$ is given by
$-K_{\P_{\D}}=\sum_{\rho \in \D(1)}D_{\rho}$. 
Hence we consider a divisor $D\in \norm{-K_{\P_{\D}}}$ which is
given by $D=\mathbf{V}(f)$ for a \emph{generic} $G$-homogeneous polynomial $f$ of $S(\Delta )$
with degree $\a=\sum_{\rho \in \D(1)}a_{\rho}$. 
\item[(c)] Let $S$ be a generic hypersurface of degree $\a$ in $D$ representing the
self-intersection class of $D\cdot D$. Compute the Chern class $c(S)$ (resp. $c(D)$) by Proposition 
$\ref{prop:chern}$ and the adjunction formula. Then the Euler characteristic $\chi(S)$ (resp. $\chi(D)$)
is determined by the 
Chern-Gauss-Bonnet formula
\begin{equation}\label{eq:Euler2}
\int_X c_n(X)=\chi(X), \qquad n=\dim_{\C}X.
\end{equation}
\item[(d)] Since $D\in\norm{-K_{\P_{\D}}}$ is a Calabi-Yau threefold, we know that $h^{0,0}(D)=h^{3,0}(D)=1$
and $h^{1,0}(D)=h^{2,0}(D)=0$ by \cite{J00}, Proposition $6.2.6$. Find all the Hodge numbers $h^{p,q}(D)$
using the Lefschetz hyperplane theorem and $\chi(D)=\sum_{p,q}(-1)^{p+q}h^{p,q}(D)$.
\item[(e)] Using the Lefschetz hyperplane theorem and Poincar\'e duality, we can calculate $H^i(S)$ for $i\neq 2$.
In order to find $h^{0,2}(S)$ and $h^{1,1}(S)$, we use Nother's formula \eqref{eq:Noether}. Since the Euler 
characteristic is also given by $\chi(S)=\sum_{p,q}(-1)^{p+q}h^{p,q}(S)$, we can check the consistency
of the values of $h^{p,q}(S)$.
\item[(f)] Calculate $\tau(S)$ by the Hodge index theorem
\begin{equation}\label{eq:Hodge_index}
\tau(S)=\sum_{p,q=0}^2 (-1)^q h^{p,q}(S)=2-h^{1,1}(S)+2 h^{0,2}(S).
\end{equation}
\item[(g)] Let $\varpi : \overline{X}\dasharrow \P_{\D}$ be the blow-up of $\P_{\D}$ along the complex 
surface $S$. Taking the proper transform of $D$ under the blow-up $\varpi$, we still denote it by $D$.
Then $\chi(\overline{X})$ is given by the formula
\[
\chi(\overline{X})=\chi(\P_{\D})-\chi(S)+\chi(E),
\]
where $E$ is the exceptional divisor of the blow-up $\varpi$. Since $E$ is a $\C P^1$-bundle over $S$,
we have
\begin{align}\label{eq:Euler}
\begin{split}
\chi(M)&=2(\chi(\overline{X})-\chi(D)) \\
&=2(\chi(\P_{\D})+\chi(S)-\chi(D)).
\end{split}
\end{align}
It is easy to compute $\chi(\P_{\D})$ using toric geometry. Thus, \eqref{eq:Euler} and the results of
(c) give $\chi(M)$.
\item[(h)] Similarly, $\tau(\overline{X})$ is given by the formula
\[
\tau(\overline{X})=\tau(\P_{\D})-\tau(S)
\]  
as the exceptional divisor has no contribution to $\tau(\overline{X})$. Thus we have
\begin{align}\label{eq:Signature}
\begin{split}
\tau(M)&=\tau(X\cup X) \\
&=2\tau(\overline{X})-\tau(D\times \C P^1)\\
&=2(\tau(\P_{\D})-\tau(S)).
\end{split}
\end{align}
We remark that $b^{2p}(\P_{\D})=h^{p,p}(\P_{\D})$ and $b^{2p+1}(\P_{\D})=0$ hold by \cite{CLS}, 
Theorem $9.3.2$. 
Then we compute $\tau(\P_{\D})$ using the Hodge index theorem
and \cite{Ful93}, p. $92$.
Substituting $\tau (S)$ in \eqref{eq:Hodge_index} and $\tau (\P_\D )$ gives $\tau (M)$.
\item[(i)] Substituting $\chi(M)$ and $\tau(M)$ into \eqref{eq:A-hat}, we conclude $\widehat{A}(M)=2$.
\end{enumerate}
The following Section $\ref{sec:Ex}$ illustrates good examples of these computations.

\section{Examples}\label{sec:Ex}
We shall compute $\widehat{A}(M)$ using the classification of toric Fano fourfolds \cite{Bat98}, \cite{Sato02}.

\begin{example}[$\C P^4$]\label{ex:CP4}\rm
Let $\D$ be the complete fan in $\bN_\R\cong \R^4$ whose $1$-dimensional cones are given by 
$\D (1)=\set{\rho_1,\rho_2,\rho_3,\rho_4,\rho_5}$ where
\begin{eqnarray*}
\rho_1=\Cone (e_1), \quad \rho_2=\Cone (e_2), \quad \rho_3=\Cone (e_3), \\
\rho_4=\Cone (e_4), \quad \text{and} \quad \rho_5=\Cone (-e_1-e_2-e_3-e_4).
\end{eqnarray*}
Then the associated toric variety is $\C P^4$ and $\PC(\D)=\set{\set{\rho_1,\rho_2,\rho_3,\rho_4,\rho_5}}$ 
(see \cite{Bat98}, \cite{Sato02}).
Hence the Stanley-Reisner ideal $\mathscr{I}_{\D}$  
is given by $\braket{x_1 x_2 x_3 x_4 x_5}$. 
Then Proposition $\ref{prop:cohomology}$ yields
\begin{align*}
H^{\bullet} (\C P^4, \Z)& \cong   \Z[x_1,x_2,x_3,x_4,x_5]/ \braket{x_1 x_2 x_3 x_4 x_5, x_1-x_5, x_2-x_5,x_3-x_5,x_4-x_5}\\
 & \cong  \Z[x_1]/\braket{x_1^5}. 
\end{align*}
The map $\bM\longrightarrow \Z^{\D(1)}$ can be written as
\begin{equation*}
\Z^4 \longrightarrow \Z^5 ,\qquad (m_1,m_2,m_3,m_4)\longmapsto (m_1,m_2,m_3,m_4,-(m_1+m_2+m_3+m_4)).
\end{equation*}
Using the map
\begin{equation*}
\Z^5 \longrightarrow \Z ,\qquad (a_1,a_2,a_3,a_4,a_5)\longmapsto a_1+a_2+a_3+a_4+a_5,
\end{equation*}
we have the exact sequence
\begin{equation*}
0 \longrightarrow \Z^4 \longrightarrow \Z^5 \longrightarrow \Z \longrightarrow 0.
\end{equation*}
Thus \eqref{diagram:Ful} implies that $A_3(\C P^4)\cong \Z$ with the generator 
$[D_1]=[D_2]=[D_3]=[D_4]=[D_5]$. Then Proposition $\ref{prop:chern}$ (i) gives $c(\C P^4)=(1+x)^5$. 
It is easy to see that the total coordinate ring  is
$S(\D)=\C [x_1,x_2,x_3,x_4,x_5]$ with $\deg(x_i)=1$. This gives a direct sum decomposition 
\[
\displaystyle S(\D)=\bigoplus_{\alpha \in \Z} S(\D)_{\alpha}.
\]

Now the anticanonical degree of $S(\D)$ is given by $\sum_{i=1}^5\deg(x_i)=5$.
Hence we consider $D\in\norm{-K_{\C P^4}}$ defined by  $D={\bf{V}}(f)$ for a generic homogeneous 
polynomial $f\in S(\D)_5$.
Then $1+c_1([D])=1+5x$ (see, Proposition $\ref{prop:chern}$ (ii)).
Therefore $c(D)=(1+x)^5(1+5x)^{-1}$, so that the Chern-Gauss-Bonnet theorem \eqref{eq:Euler2} implies
\[
\chi(D)=\braket{c_3(D), [D]}=\braket{-40x^3, [5x]}=-200,
\]
where we used the normalization $\int_{\C P^4}x^4=1$.
Especially the Lefschetz hyperplane theorem implies $h^{1,1}(D)=b^2(\P_{\D})=1$. As $D$ is a Calabi-Yau divisor,
we also see that $h^{3,0}(D)=h^{0,0}(D)=1$, whence $h^{2,1}(D)=\frac{1}{2}(2h^{1,1}(D)-\chi(D))=101$.

Let $S$ be a generic smooth complex surface in $D$ representing the self-intersection class of $D\cdot D$. The Chern class
$c(S)=(1+x)^5(1+5x)^{-2}$ determines the Euler characteristic by
\[
\chi(S)=\braket{c_2(S), [S]}=\braket{35x^2, [(5x)^2]}=875.
\]
Remark that $b^1(S)=0$ as $S$ is simply-connected. In order to find the cohomology of the middle dimension 
$H^{p,q}(S) \;\; (p+q=2)$, we will use Noether's formula as follows. 
Since $h^{0,0}(S)=1$ and $h^{0,1}(S)=0$, Theorem $\ref{th:Noether}$ implies that
\[
12(1+h^{0,2}(S))=\braket{35x^2+25x^2, [(5x)^2]}.
\]
Hence we conclude that $h^{0,2}(S)=124$. Similarly we have 
\[
6(-h^{1,1}(S))=\braket{25x^2-5\cdot 35x^2, [(5x)^2]}.
\]
Then $h^{1,1}(S)=625$. Summing up these computations, we conclude
 \begin{equation*}
 h^{p,q}(S) =
\begin{array}{ccccc}
&& 1 && \\
& 0 && 0 & \\
124 && 625 && 124 \\
& 0 && 0 & \\
&& 1 && \\
\end{array}.
\end{equation*}
This result is consistent with the value $\chi(S)=875$. 
By the Hodge index theorem, we find the signature of $S$ is $\tau(S)=-375$.

Finally we shall show that the resulting $8$-manifold $M$ admits a metric with holonomy $\SU(4)$.
By \eqref{eq:Euler}, we have $\chi(M)=2(5+875-(-200))=2160$. Meanwhile, $\tau(M)=2(1-(-375))=752$ by \eqref{eq:Signature}.
Therefore \eqref{eq:A-hat} gives $\widehat{A}(M)=2$. The assertion follows from Corollary $\ref{cor:doubling}$.
\end{example}

\begin{example}[$B_1$]\rm
Let $\D$ be the complete fan in $\bN_\R\cong \R^4$ whose $1$-dimensional cones are given by 
$\D (1)=\set{\rho_1,\dots, \rho_6}$ where
\begin{eqnarray*}
\rho_1=\Cone (e_4), \quad \rho_2=\Cone (e_1), \quad \rho_3=\Cone (e_2), \quad \rho_4=\Cone (e_3),\\
\rho_5=\Cone (-e_4), \quad \text{and} \quad \rho_6=\Cone (-e_1-e_2-e_3+3e_4).
\end{eqnarray*}
Then we readily see that $\P_{\D}=\P(\mathcal{O}_{\C P^3}\oplus \mathcal{O}_{\C P^3}(3))$, 
$\PC(\D)=\set{\set{\rho_1,\rho_5}, \set{\rho_2, \rho_3, \rho_4, \rho_6}}$ and 
$\mathscr{I}_{\D}=\braket{x_1x_5, x_2x_3x_4x_6}$.
Therefore, Proposition $\ref{prop:cohomology}$ implies that
\begin{equation}\label{eq:B2_cohomology}
H^{\bullet} (\P_{\D}, \Z) \cong   \Z[x,y]/\braket{x(x+3y), y^4}. 
\end{equation}
The map $\bM\longrightarrow \Z^{\D(1)}$ is defined by
\begin{equation*}
\Z^4 \longrightarrow \Z^6 ,\qquad (m_1,\dots, m_6)\longmapsto (m_4,m_1,m_2,m_3,-m_4,-m_1-m_2-m_3+3m_4).
\end{equation*}
Meanwhile, the Chow group is generated by the classes of $D_i$ with the relations
\begin{align*}
0\sim \divisor(\chi^{e_1})&=D_2-D_6, &\quad  0\sim \divisor(\chi^{e_2})&=D_3-D_6, \\
0\sim \divisor(\chi^{e_3})&=D_4-D_6, &\quad  0\sim \divisor(\chi^{e_4})&=D_1-D_5+3D_6,
\end{align*}
by \eqref{diagram:Ful}. 
Hence we conclude that $A_3(\P_{\D})\cong \Z^2$ with generators 
$[D_2]=[D_3]=[D_4]=[D_6]$ and $[D_5]=[D_1]+3[D_6]$. Then Proposition $\ref{prop:chern}$ (i) gives 
$c(\P_{\D})=(1+x)(1+y)^4(1+x+3y)$. In this case, the map $\Z^{\D(1)}=\Z^6\longrightarrow \Z^2=A_3(\P_{\D})$ is
$(a_1,\dots, a_6)\longmapsto (a_1+a_5, a_2+a_3+a_4+3a_5+a_6)$. This gives the grading on
$S(\D)=\C[x_1,\dots,x_6]$, where 
\begin{align*}
\deg(x_1)&=(1,0), \qquad  \deg(x_5)=(1,3), \\
\deg(x_2)&=\deg(x_3)=\deg(x_4)=\deg(x_6)=(0,1).
\end{align*}
In particular, $S(\D)$ is a bihomogeneous polynomial ring.

Now the anticanonical degree of $S(\D)$ is $\a:=\sum_{i=1}^6\deg(x_i)=(2,7)$.
Taking a generic homogeneous polynomial $f\in S(\D)_{\a}$, we consider a torus-invariant anticanonical
divisor $D={\bf{V}}(f)$ on $\P_{\D}$.
Since $1+c_1([D])=1+(2x+7y)$, we have
$c(D)=(1+x)(1+y)^4(1+x+3y)(1+2x+7y)^{-1}$, so that
\[
\chi(D)=\braket{c_3(D), [D]} =\braket{(-4)\cdot(26y^3+8xy^2), [2x+7y]}=-240,
\]
where we used the normalization $\int_{B_1}xy^3=1$.
The Lefschetz hyperplane theorem gives $h^{1,1}(D)=b^2(\P_{\D})=2$, whence $h^{2,1}(D)=122$.

Let $S$ be a generic ample hypersurface of degree $\a$ in $D$ as in Example $\ref{ex:CP4}$. Since
$c(S)=(1+x)(1+y)^4(1+x+3y)(1+2x+7y)^{-2}$, we have $c_2(S)=67y^2+36xy+4x^2=67y^2+24xy$ where we used
the relation\footnote{In the practical computation we used packages (a) Macaulay2 and (b) Maxima. 
These open source algebra systems are available at {\tt{http://www.math.uiuc.edu/Macaulay2}} and {\tt{http://maxima.sourceforge.net}} respectively. } 
in the cohomology ring \eqref{eq:B2_cohomology}. Hence 
\[
\chi(S)=\braket{67y^2+24xy, [(2x+7y)^2]}=1096.
\]
Now Theorem $\ref{th:Noether}$ implies that
$12(1+h^{0,2}(S))=1896$ and $6h^{1,1}(S)=4680$.  
Consequently we have
 \begin{equation*}
 h^{p,q}(S) =
\begin{array}{ccccc}
&& 1 && \\
& 0 && 0 & \\
157 && 780 && 157 \\
& 0 && 0 & \\
&& 1 && \\
\end{array}.
\end{equation*}
Then the Hodge index theorem gives $\tau(S)=-464$. Thus $\tau(M)=2(0-(-464))=928$ by \eqref{eq:Signature}.
Also, \eqref{eq:Euler} gives $\chi(M)=2(8+1096-(-240))=2688$. Hence $\widehat{A}(M)=\frac{1}{48}(3\cdot 928
-2688)=2$ by \eqref{eq:A-hat}.
\end{example}

\newpage

\section{Table of examples from toric Fano fourfolds}
In the following table we give the list of all Calabi-Yau fourfolds constructed in Proposition $\ref{prop:CY4}$.
In the table below `ID' denotes the database ID in \cite{graded_ring}, 
and $(\chi (M),\tau (M))$ denotes the pair of the Euler characteristic and the signature 
of the resulting Calabi-Yau fourfold $M$. In the last column
we indicate the same notation of toric Fano fourfolds as in \cite{Bat98} and \cite{Sato02}.

\begin{table}[H]
\caption{All possible Calabi-Yau fourfolds from toric Fano fourfolds} \vspace{0.3cm}
\begin{center} 
\begin{tabular}{rrcc} \toprule
No. & ID  & ($\chi(M),\tau(M))$ & Notation \\  \midrule
 $1$    &  $147$    &   $(2160,752)$ & $\C P^4$  \\ 
 $2$    &  $25$    &   $(2688,928)$ & $B_1$  \\ 
 $3$    &  $139$    &   $(2208,768)$ & $B_2$  \\
 $4$    &  $144$    &   $(1920,672)$ & $B_3$  \\
 $5$    &  $145$    &   $(1824,640)$ & $B_4$  \\
 $6$    &  $138$    &   $(1824,640)$ & $B_5$  \\
 $7$    &  $44$    &   $(2058,718)$ & $C_1$  \\
 $8$    &  $141$    &   $(1824,640)$ & $C_2$  \\
 $9$    &  $70$    &   $(1824,640)$ & $C_3$  \\
 $10$  &  $146$    &   $(1746,614)$ & $C_4$  \\
 $11$  &  $24$    &   $(2124,740)$ & $E_1$  \\
 $12$  &  $128$    &   $(1764,620)$ & $E_2$  \\
 $13$  &  $127$    &   $(1584,560)$ & $E_3$  \\
 $14$  &  $30$    &   $(2064,720)$ & $D_1$  \\
 $15$  &  $31$    &   $(2016,704)$ & $D_2$  \\
 $16$  &  $49$    &   $(1968,688)$ & $D_3$  \\
 $17$  &  $35$    &   $(1968,688)$ & $D_4$  \\
 $18$  &  $42$    &   $(1776,624)$ & $D_5$  \\
 $19$  &  $129$    &   $(1776,624)$ & $D_6$  \\
 $20$  &  $97$    &   $(1740,612)$ & $D_7$  \\
 $21$  &  $134$    &   $(1728,608)$ & $D_8$  \\
 $22$  &  $66$    &   $(1680,592)$ & $D_9$  \\
 $23$  &  $132$    &   $(1680,592)$ & $D_{10}$  \\
 $24$  &  $117$    &   $(1662,586)$ & $D_{11}$  \\
 $25$  &  $140$    &   $(1632,576)$ & $D_{12}$  \\
 $26$  &  $143$    &   $(1584,560)$ & $D_{13}$  \\
 $27$  &  $133$    &   $(1584,560)$ & $D_{14}$  \\
 $28$  &  $135$    &   $(1584,560)$ & $D_{15}$  \\
 $29$  &  $68$      &   $(1584,560)$ & $D_{16}$  \\
 $30$  &  $109$    &   $(1506,534)$ & $D_{17}$  \\
 $31$  &  $43$      &   $(1488,528)$ & $D_{18}$  \\
 $32$  &  $136$    &   $(1488,528)$ & $D_{19}$  \\
 $33$  &  $41$      &   $(1872,656)$ & $G_{1}$  \\
 $34$  &  $40$      &   $(1626,574)$ & $G_{2}$  \\
 $35$  &  $64$      &   $(1584,560)$ & $G_{3}$  \\
 $36$  &  $60$      &   $(1536,544)$ & $G_{4}$  \\
 $37$  &  $69$      &   $(1506,534)$ & $G_{5}$  \\
 $38$  &  $137$    &   $(1488,528)$ & $G_{6}$  \\
 $39$  &  $26$      &   $(1974,690)$ & $H_{1}$  \\
$40$  &  $45$      &   $(1812,636)$ & $H_{2}$  \\
$41$  &  $28$      &   $(1734,610)$ & $H_{3}$  \\
\bottomrule
\end{tabular} \hfill
\begin{tabular}{rrcc} \toprule
No. & ID  & ($\chi(M),\tau(M))$ & Notation \\  \midrule
$42$  &  $118$    &   $(1632,576)$ & $H_{4}$  \\
$43$  &  $123$    &   $(1536,544)$ & $H_{5}$  \\
$44$  &  $48$      &   $(1524,540)$ & $H_{6}$  \\
$45$  &  $32$      &   $(1446,514)$ & $H_{7}$  \\
$46$  &  $124$    &   $(1422,506)$ & $H_{8}$  \\
$47$  &  $125$    &   $(1392,496)$ & $H_{9}$  \\
$48$  &  $67$      &   $(1344,480)$ & $H_{10}$  \\
$49$  &  $74$      &   $(1728,608)$ & $L_{1}$  \\
$50$  &  $75$      &   $(1680,592)$ & $L_{2}$  \\
$51$  &  $83$      &   $(1632,576)$ & $L_{3}$  \\
$52$  &  $105$    &   $(1584,560)$ & $L_{4}$  \\
$53$  &  $95$      &   $(1536,544)$ & $L_{5}$  \\
$54$  &  $112$      &   $(1488,528)$ & $L_{6}$  \\
$55$  &  $106$      &   $(1440,512)$ & $L_{7}$  \\
$56$  &  $142$      &   $(1440,512)$ & $L_{8}$  \\
$57$  &  $130$      &   $(1440,512)$ & $L_{9}$  \\
$58$  &  $114$      &   $(1440,512)$ & $L_{10}$  \\
$59$  &  $131$      &   $(1344,480)$ & $L_{11}$  \\
$60$  &  $108$      &   $(1344,480)$ & $L_{12}$  \\
$61$  &  $96$      &   $(1344,480)$ & $L_{13}$  \\
$62$  &  $33$      &   $(1776,624)$ & $I_{1}$  \\
$63$  &  $29$      &   $(1686,594)$ & $I_{2}$  \\
$64$  &  $47$      &   $(1620,572)$ & $I_{3}$  \\
$65$  &  $38$      &   $(1578,558)$ & $I_{4}$  \\
$66$  &  $34$      &   $(1542,546)$ & $I_{5}$  \\
$67$  &  $93$      &   $(1518,538)$ & $I_{6}$  \\
$68$  &  $37$      &   $(1488,528)$ & $I_{7}$  \\
$69$  &  $115$      &   $(1440,512)$ & $I_{8}$  \\
$70$  &  $94$      &   $(1452,516)$ & $I_{9}$  \\
$71$  &  $111$      &   $(1458,518)$ & $I_{10}$  \\
$72$  &  $59$      &   $(1440,512)$ & $I_{11}$  \\
$73$  &  $116$      &   $(1326,474)$ & $I_{12}$  \\
$74$  &  $126$      &   $(1392,496)$ & $I_{13}$  \\
$75$  &  $104$      &   $(1362,486)$ & $I_{14}$  \\
$76$  &  $39$      &   $(1290,462)$ & $I_{15}$  \\
$77$  &  $61$      &   $(1440,512)$ & $M_{1}$  \\
$78$  &  $50$      &   $(1536,544)$ & $M_{2}$  \\
$79$  &  $58$      &   $(1392,496)$ & $M_{3}$  \\
$80$  &  $57$      &   $(1392,496)$ & $M_{4}$  \\
$81$  &  $110$      &   $(1374,490)$ & $M_{5}$  \\
$82$  &  $36$      &   $(1398,498)$ & $J_{1}$  \\
\bottomrule
\end{tabular}
\end{center}
\end{table}

\begin{table}[H]
\begin{center} 
\begin{tabular}{rrcc} \toprule
No. & ID  & ($\chi(M),\tau(M))$ & Notation \\  \midrule
$83$  &  $65$      &   $(1266,454)$ & $J_{2}$  \\
$84$  &  $71$      &   $(1620,572)$ & $Q_{1}$  \\
$85$  &  $79$      &   $(1506,534)$ & $Q_{2}$  \\
$86$  &  $73$      &   $(1476,524)$ & $Q_{3}$  \\
$87$  &  $77$      &   $(1506,534)$ & $Q_{4}$  \\
$88$  &  $81$      &   $(1410,502)$ & $Q_{5}$  \\
$89$  &  $84$      &   $(1392,496)$ & $Q_{6}$  \\
$90$  &  $91$      &   $(1374,490)$ & $Q_{7}$  \\
$91$  &  $90$      &   $(1344,480)$ & $Q_{8}$  \\
$92$  &  $82$      &   $(1314,470)$ & $Q_{9}$  \\
$93$  &  $102$      &   $(1296,464)$ & $Q_{10}$  \\
$94$  &  $120$      &   $(1296,464)$ & $Q_{11}$  \\
$95$  &  $88$      &   $(1278,458)$ & $Q_{12}$  \\
$96$  &  $76$      &   $(1284,460)$ & $Q_{13}$  \\
$97$  &  $103$      &   $(1266,454)$ & $Q_{14}$  \\
$98$  &  $113$      &   $(1248,448)$ & $Q_{15}$  \\
$99$  &  $92$      &   $(1212,436)$ & $Q_{16}$  \\
$100$  &  $107$      &   $(1182,426)$ & $Q_{17}$  \\
$101$  &  $27$      &   $(1404,500)$ & $K_{1}$  \\
$102$  &  $46$      &   $(1368,488)$ & $K_{2}$  \\
$103$  &  $119$    &   $(1296,464)$ & $K_{3}$  \\
$104$  &  $122$    &   $(1260,452)$ & $K_{4}$  \\
$105$  &  $89$      &   $(1278,458)$ & $R_{1}$  \\
$106$  &  $51$      &   $(1248,448)$ & $R_{2}$  \\
$107$  &  $56$      &   $(1200,432)$ & $R_{3}$  \\
$108$  &  $52$      &   $(1218,438)$ &           \\
$109$  &  $72$      &   $(1224,440)$ & $U_{1}$  \\
$110$  &  $80$      &   $(1188,428)$ & $U_{2}$  \\
$111$  &  $78$      &   $(1188,428)$ & $U_{3}$  \\
$112$  &  $101$    &  $(1152,416)$ & $U_{4}$  \\
$113$  &  $121$    &  $(1152,416)$ & $U_{5}$  \\
$114$  &  $85$      &  $(1152,416)$ & $U_{6}$  \\
$115$  &  $86$      &  $(1116,404)$ & $U_{7}$  \\
$116$  &  $87$      &  $(1080,392)$ & $U_{8}$  \\
$117$  &  $62$      &  $(1200,432)$ & $\widetilde{V}^4$  \\
$118$  &  $63$      &  $(960,352)$ & ${V}^4$  \\
$119$  &  $98$      &  $(1170,422)$ & $S_2\times S_2$  \\
$120$  &  $99$      &  $(1044,380)$ & $S_2\times S_3$  \\
$121$  &  $100$      &  $(936,344)$ & $S_3\times S_3$  \\
$122$  &  $55$      &  $(1248,448)$ & $Z_1$  \\
$123$  &  $53$      &  $(1266,454)$ & $Z_2$  \\
$124$  &  $54$      &  $(1026,374)$ & $W$  \\
\bottomrule
\end{tabular}
\end{center}
\end{table}

\end{document}